\documentclass[11pt]{article}

\usepackage[T1]{fontenc}
\usepackage[utf8]{inputenc}
\usepackage{lmodern}
\usepackage[margin=1in]{geometry}
\usepackage{microtype}

\usepackage{amsmath,amssymb,amsthm,mathtools}
\usepackage{bm}

\usepackage{graphicx}
\graphicspath{{figures/}}
\usepackage{booktabs}
\usepackage{array}

\usepackage[shortlabels]{enumitem}
\usepackage[numbers,sort&compress]{natbib}
\usepackage{xcolor}
\usepackage[colorlinks=true,linkcolor=blue,citecolor=blue,urlcolor=blue]{hyperref}
\usepackage[capitalise,noabbrev]{cleveref}

\theoremstyle{plain}
\newtheorem{theorem}{Theorem}[section]
\newtheorem{lemma}[theorem]{Lemma}

\newtheorem{corollary}[theorem]{Corollary}

\theoremstyle{definition}
\newtheorem{definition}[theorem]{Definition}

\theoremstyle{remark}

\newcommand{\R}{\mathbb{R}}

\newcommand{\E}{\mathbb{E}}
\newcommand{\Pp}{\mathbb{P}}
\newcommand{\Sph}{\mathbb{S}}
\newcommand{\one}{\mathbf{1}}
\newcommand{\norm}[1]{\left\lVert #1\right\rVert}
\newcommand{\ip}[2]{\left\langle #1,#2\right\rangle}
\newcommand{\abs}[1]{\left\lvert #1\right\rvert}

\DeclareMathOperator{\cone}{cone}
\DeclareMathOperator{\range}{range}

\title{Small-Ball Marginals Do Not Control Restricted Eigenvalues\\
by Euclidean Gaussian Width}
\author{Jinze Zhao\\University of California, San Diego\\\texttt{jiz419@ucsd.edu}}
\date{}

\begin{document}

\maketitle

% \begin{center}
%   \small\emph{Research draft. This manuscript has not been peer reviewed;
%   the literature-status statement is necessarily provisional.}
% \end{center}

\begin{abstract}
Banerjee, Chen, and Sivakumar asked at COLT 2015 whether a uniform
small-ball condition on the rows of a random design matrix forces a
restricted-eigenvalue lower bound whose sample complexity is governed by
the ordinary Euclidean Gaussian width of an arbitrary spherical subset.
We give a negative answer to the natural distribution-free formulation of
that question.  For every sample size $n$, we construct a centered,
genuinely heavy-tailed row distribution in dimension $p=4^n+1$ and a set
$A=C\cap\Sph^{p-1}$, where $C$ is a closed polyhedral convex cone, such
that
\[
  \inf_{v\ne 0}\Pp\!\left(
    \abs{\ip{Z}{v}}\ge \frac{\norm{v}_2}{\sqrt{2}}
  \right)\ge \frac1{12}
  \quad\text{and}\quad
  w(A)<2.
\]
Nevertheless, for the matrix $X$ with $n$ independent copies of $Z$ as
rows,
\[
  \Pp\!\left(\inf_{u\in A}\norm{Xu}_2^2=0\right)
  \ge 1-\exp(-2^n).
\]
Thus no positive constants depending only on the fixed small-ball
parameters can yield a lower bound of the proposed form
$c_1n-c_2w(A)^2$ with high probability.  The construction isolates the
obstruction: a marginal small-ball lower bound controls every fixed
direction, but does not control the distribution-dependent complexity of
searching over many directions.  We state the quantifiers explicitly and
discuss why isotropic or upper-tail assumptions lead to a different,
still meaningful problem.

\end{abstract}

% \newpage
% \tableofcontents
% \newpage

\section{Introduction}
\label{sec:introduction}

Restricted-eigenvalue inequalities are a basic route from random-design
assumptions to guarantees for sparse and structured estimators.  If
$X\in\R^{n\times p}$ and $A\subset\Sph^{p-1}$ is a restricted set of
directions, the central quantity is
\[
  \inf_{u\in A}\norm{Xu}_2^2.
\]
For sub-Gaussian rows, generic-chaining arguments connect this quantity to
the Euclidean Gaussian width
\[
  w(A):=\E_g\sup_{u\in A}\ip{g}{u},
  \qquad g\sim N(0,I_p).
\]
This motivates bounds of the schematic form
\begin{equation}
  \inf_{u\in A}\norm{Xu}_2^2
  \ge c_1 n-c_2 w(A)^2.
  \label{eq:schematic-bound}
\end{equation}

In a COLT 2015 open-problem note, \citet{banerjee2015open} asked whether
the same Gaussian-width sample complexity follows for heavy-tailed rows
from a small-ball assumption alone.  In their notation, a row $Z$ obeys
the small-ball condition on a set $E\subset\R^p$ if there are
$\alpha,\beta>0$ such that
\begin{equation}
  \inf_{v\in E}
  \Pp\bigl(\abs{\ip{Z}{v}}\ge \alpha\norm{v}_2\bigr)
  \ge \beta.
  \label{eq:intro-small-ball}
\end{equation}
The note defines a ``spherical cap'' simply as a subset of the unit
sphere and asks for \eqref{eq:schematic-bound} for every such set.  It
does not impose isotropy, independence of coordinates, or an upper-tail
condition on $Z$.

The answer under these stated assumptions is no.  Our counterexample has
three interacting ingredients.  First, the restricted directions are
many almost-parallel rays, so their conic Gaussian width stays bounded as
their number grows exponentially.  Second, an explicit dual frame lets
one coordinate label control the measurement in each ray exactly.  Third,
each label is independently absent with probability $1/2$ in each row.
Every fixed direction is still visible with a constant probability---in
fact, a fourth-moment calculation gives a uniform small-ball condition on
all of $\R^p$---but among exponentially many rays, one ray is absent from
all $n$ rows with overwhelming probability.

Our main contribution is an explicit negative theorem with fixed
constants $\alpha=1/\sqrt2$ and $\beta=1/12$.  The restricted set is not
merely finite: it is $C\cap\Sph^{p-1}$ for a closed, finitely generated
convex cone $C$.  The row law is centered, has finite covariance and
polynomial tails, and every standard coordinate has infinite fourth
moment.  All estimates are elementary and quantitative.

This result does not contradict the small-ball method of
\citet{mendelson2015learning,koltchinskii2015bounding,tropp2015convex}.
The general method involves a \emph{mean empirical width} adapted to the
row distribution, rather than the Euclidean Gaussian width alone.  Under
sub-Gaussian increment assumptions these complexities can be compared;
the construction below shows that a marginal small-ball lower bound does
not provide that comparison.

\paragraph{Status and scope.}
Targeted searches through August 16, 2026 located no publication that
states this counterexample or explicitly resolves the 2015 question.
That is negative evidence, not a certification of novelty.  More
importantly, the phrase ``suitable constants'' in the original note is
informal.  We therefore formalize and refute the natural high-dimensional
reading in which the constants are controlled by the fixed small-ball
parameters uniformly over the dimension, the restricted set, and the row
law.  If the constants may depend arbitrarily on all of those objects and
on $n$, the proposed sample-complexity statement has no uniform content.

The remainder of the paper states this interpretation precisely,
constructs the example, proves each estimate, and then records the limits
of the negative result.

\section{The question and its uniform interpretation}
\label{sec:problem}

Throughout, $\Sph^{p-1}:=\{u\in\R^p:\norm{u}_2=1\}$.  We use the
one-sided Gaussian-width convention in the original open problem:
\begin{equation}
  w(A):=\E\sup_{u\in A}\ip{g}{u},
  \qquad g\sim N(0,I_p).
  \label{eq:gaussian-width}
\end{equation}

\begin{definition}[Uniform small-ball-to-Gaussian-width assertion]
\label{def:uniform-assertion}
Fix $\alpha,\beta>0$.  The uniform assertion says that there are
$c_1=c_1(\alpha,\beta)>0$, $c_2=c_2(\alpha,\beta)>0$, and a sequence
$q_n=q_n(\alpha,\beta)\to0$ such that the following holds simultaneously
for every $n,p$, every random vector $Z\in\R^p$ satisfying
\begin{equation}
  \inf_{v\in\R^p\setminus\{0\}}
  \Pp\bigl(\abs{\ip{Z}{v}}\ge\alpha\norm{v}_2\bigr)
  \ge\beta,
  \label{eq:global-small-ball}
\end{equation}
and every $A\subset\Sph^{p-1}$.  If the rows of $X\in\R^{n\times p}$
are independent copies of $Z$, then
\begin{equation}
  \Pp\left(
    \inf_{u\in A}\norm{Xu}_2^2
    \ge c_1n-c_2w(A)^2
  \right)\ge 1-q_n.
  \label{eq:uniform-assertion}
\end{equation}
\end{definition}

Condition \eqref{eq:global-small-ball} is stronger than necessary for
the original question, which only postulates the condition on some set
$E$.  Using the global form removes any ambiguity about the relation
between $E$ and $A$.  The essential point in
\cref{def:uniform-assertion} is that $c_1,c_2$ are determined by the
fixed small-ball parameters, not by the dimension, the set $A$, the
sample size, or otherwise uncontrolled features of the row law.

Our result refutes this assertion at one fixed pair of small-ball
constants.

\begin{theorem}[Counterexample]
\label{thm:counterexample}
Set $\alpha_0=1/\sqrt2$ and $\beta_0=1/12$.  For every integer $n\ge1$
there exist $p=4^n+1$, a closed polyhedral convex cone $C\subset\R^p$,
$A=C\cap\Sph^{p-1}$, and a centered random vector $Z\in\R^p$ such that:
\begin{enumerate}[(i)]
  \item $Z$ has finite covariance, polynomial tails, and an infinite
  fourth moment in every standard coordinate;
  \item $Z$ satisfies \eqref{eq:global-small-ball} with
  $(\alpha,\beta)=(\alpha_0,\beta_0)$;
  \item $w(A)<2$; and
  \item if $X$ has $n$ independent copies of $Z$ as rows, then
  \begin{equation}
    \Pp\left(\inf_{u\in A}\norm{Xu}_2^2=0\right)
    \ge 1-\exp(-2^n).
    \label{eq:theorem-failure}
  \end{equation}
\end{enumerate}
\end{theorem}

\begin{corollary}[Negative answer to the uniform formulation]
\label{cor:negative-answer}
The assertion in \cref{def:uniform-assertion} is false for
$(\alpha,\beta)=(1/\sqrt2,1/12)$.  More concretely, for every fixed
$c_1,c_2>0$, the proposed inequality fails for the examples in
\cref{thm:counterexample} with probability at least
$1-\exp(-2^n)$ whenever $n>4c_2/c_1$.
\end{corollary}

We now construct the example and prove the theorem from first principles.

\section{Construction}
\label{sec:construction}

Fix an integer $m\ge2$ for the moment and work in $\R^{m+1}$ with
standard orthonormal basis $e_0,e_1,\ldots,e_m$.  Define
\begin{equation}
  \delta:=\frac1{\sqrt{\log(2m)}}
  \quad\text{and}\quad
  u_j:=\frac{e_0+\delta e_j}{\sqrt{1+\delta^2}},
  \qquad 1\le j\le m.
  \label{eq:directions}
\end{equation}
Let $U=[u_1\ \cdots\ u_m]\in\R^{(m+1)\times m}$ and put
\begin{equation}
  G:=U^\top U,
  \qquad
  V:=UG^{-1},
  \qquad
  h:=\frac{\delta e_0-\sum_{j=1}^m e_j}{\sqrt{m+\delta^2}}.
  \label{eq:dual-and-normal}
\end{equation}
The inverse in \eqref{eq:dual-and-normal} exists by
\cref{lem:linear-algebra} below.  The restricted cone and its spherical
section are
\begin{equation}
  C:=\cone\{u_1,\ldots,u_m\},
  \qquad A:=C\cap\Sph^m.
  \label{eq:cone-set}
\end{equation}

The random row is defined next.  Let $Y_1,\ldots,Y_m$ be independent
with
\begin{equation}
  \Pp(Y_j=0)=\frac12,
  \qquad
  \Pp(Y_j=1)=\Pp(Y_j=-1)=\frac14.
  \label{eq:Y-law}
\end{equation}
Write $Y=(Y_1,\ldots,Y_m)^\top$, and let $\eta$ be an independent
Rademacher random variable.  Finally, let $H$ be independent of
$(Y,\eta)$ with the Pareto tail
\begin{equation}
  \Pp(H\ge t)=t^{-3},\qquad t\ge1.
  \label{eq:H-law}
\end{equation}
Thus $H\ge1$ almost surely, $\E H^2<\infty$, and $\E H^4=\infty$.
Set
\begin{equation}
  B:=\sqrt{2m}\,VY+\eta h,
  \qquad
  Z:=HB.
  \label{eq:row-law}
\end{equation}

Two features of this definition should be noted.  The dual matrix $V$
makes the $j$th label $Y_j$ exactly observable in direction $u_j$.
The normal term $\eta h$ is needed to keep the small-ball estimate
nondegenerate in the one-dimensional complement of $\range(U)$.

\section{Proofs}
\label{sec:proofs}

\subsection{Linear algebra of the frame}

\begin{lemma}[Gram matrix, duality, and orthogonal decomposition]
\label{lem:linear-algebra}
The matrix $U$ has rank $m$, and
\begin{equation}
  G=\frac{\one\one^\top+\delta^2 I_m}{1+\delta^2}.
  \label{eq:gram}
\end{equation}
The eigenvalues of $G$ are
\begin{equation}
  \frac{m+\delta^2}{1+\delta^2}
  \quad\text{once},
  \qquad
  \frac{\delta^2}{1+\delta^2}
  \quad\text{with multiplicity }m-1.
  \label{eq:gram-eigenvalues}
\end{equation}
Consequently, $G$ is positive definite and $G\preceq mI_m$.  Moreover,
$h$ is a unit vector orthogonal to $\range(U)$,
\begin{equation}
  U^\top V=V^\top U=I_m,
  \label{eq:dual-identity}
\end{equation}
and every $v\in\R^{m+1}$ has a unique representation
\begin{equation}
  v=Ua+bh,
  \qquad a\in\R^m,\quad b\in\R,
  \label{eq:decomposition}
\end{equation}
for which
\begin{equation}
  \norm{v}_2^2=a^\top Ga+b^2.
  \label{eq:v-norm}
\end{equation}
\end{lemma}

\begin{proof}
From \eqref{eq:directions},
\[
  \ip{u_j}{u_k}
  =\frac{1+\delta^2\mathbf{1}_{\{j=k\}}}{1+\delta^2}.
\]
Thus every off-diagonal entry of $G$ is $1/(1+\delta^2)$ and every
diagonal entry is $1$, which proves \eqref{eq:gram}.  The matrix
$\one\one^\top$ acts as multiplication by $m$ on $\range(\one)$ and
as zero on $\one^\perp$.  Adding $\delta^2I_m$ and dividing by
$1+\delta^2$ gives the eigenvalues in
\eqref{eq:gram-eigenvalues}.  They are strictly positive, so $G$ is
invertible and $U$ has rank $m$.

The largest eigenvalue is at most $m$ because
\[
  \frac{m+\delta^2}{1+\delta^2}\le m
  \quad\Longleftrightarrow\quad
  \delta^2\le m\delta^2,
\]
which holds for $m\ge1$.  Hence $G\preceq mI_m$.

Next, using \eqref{eq:directions} and \eqref{eq:dual-and-normal}, for
each $j$ we obtain
\[
  \ip{u_j}{h}
  =\frac{\delta-\delta}
  {\sqrt{(1+\delta^2)(m+\delta^2)}}=0.
\]
Also,
\[
  \norm{h}_2^2
  =\frac{\delta^2+m}{m+\delta^2}=1.
\]
Therefore $h$ is a unit normal to the $m$-dimensional space
$\range(U)$.  Since $V=UG^{-1}$ and $G=G^\top$,
\[
  U^\top V=U^\top UG^{-1}=GG^{-1}=I_m
\]
and, on taking transposes, $V^\top U=I_m$.  The orthogonal direct sum
$\range(U)\oplus\range(h)=\R^{m+1}$ gives the existence and uniqueness
of \eqref{eq:decomposition}.  Finally, $U^\top h=0$ and
$\norm{h}_2=1$ imply
\[
  \norm{Ua+bh}_2^2
  =a^\top U^\top Ua+b^2
  =a^\top Ga+b^2,
\]
which is \eqref{eq:v-norm}.
\end{proof}

The next identity is the mechanism behind the failure event.

\begin{lemma}[Exact directional labels]
\label{lem:labels}
For every $1\le j\le m$,
\begin{equation}
  \ip{Z}{u_j}=H\sqrt{2m}\,Y_j.
  \label{eq:label-identity}
\end{equation}
\end{lemma}

\begin{proof}
By \eqref{eq:dual-identity}, $V^\top u_j$ is the $j$th standard basis
vector of $\R^m$, and by \cref{lem:linear-algebra},
$\ip{h}{u_j}=0$.  Therefore
\[
  \ip{Z}{u_j}
  =H\left(\sqrt{2m}\,Y^\top V^\top u_j
  +\eta\ip{h}{u_j}\right)
  =H\sqrt{2m}\,Y_j.
\]
\end{proof}

\subsection{The uniform small-ball estimate}

For completeness, we include the elementary form of the
Paley--Zygmund inequality used below.

\begin{lemma}[Paley--Zygmund inequality]
\label{lem:paley-zygmund}
If $Q\ge0$, $0<\E Q<\infty$, $\E Q^2<\infty$, and
$0<\theta<1$, then
\begin{equation}
  \Pp(Q\ge\theta\E Q)
  \ge (1-\theta)^2\frac{(\E Q)^2}{\E Q^2}.
  \label{eq:PZ}
\end{equation}
\end{lemma}

\begin{proof}
Split the expectation according to the event
$D:=\{Q\ge\theta\E Q\}$.  Since $Q<\theta\E Q$ on $D^c$,
\[
  \E Q
  =\E[Q\mathbf{1}_{D^c}]+\E[Q\mathbf{1}_D]
  \le \theta\E Q+\E[Q\mathbf{1}_D].
\]
Rearranging and then applying Cauchy--Schwarz gives
\[
  (1-\theta)\E Q
  \le\E[Q\mathbf{1}_D]
  \le(\E Q^2)^{1/2}\Pp(D)^{1/2}.
\]
Squaring both sides and dividing by $\E Q^2$ proves
\eqref{eq:PZ}.
\end{proof}

\begin{lemma}[Dimension-free small-ball constants]
\label{lem:small-ball}
For the random vector $Z$ in \eqref{eq:row-law},
\begin{equation}
  \inf_{v\in\R^{m+1}\setminus\{0\}}
  \Pp\left(
    \abs{\ip{Z}{v}}\ge\frac{\norm{v}_2}{\sqrt2}
  \right)\ge\frac1{12}.
  \label{eq:proved-small-ball}
\end{equation}
The constants do not depend on $m$.
\end{lemma}

\begin{proof}
Fix $v\ne0$ and use the unique representation $v=Ua+bh$ from
\cref{lem:linear-algebra}.  Set
\begin{equation}
  R:=a^\top Y=\sum_{j=1}^m a_jY_j,
  \qquad
  W:=\sqrt{2m}\,R,
  \qquad
  S:=W+b\eta.
  \label{eq:RWS}
\end{equation}
Equations \eqref{eq:row-law}, \eqref{eq:dual-identity}, and
$U^\top h=0$ give
\begin{equation}
  \ip{B}{v}
  =\sqrt{2m}\,Y^\top V^\top Ua+b\eta
  =\sqrt{2m}\,a^\top Y+b\eta
  =S,
  \qquad
  \ip{Z}{v}=HS.
  \label{eq:inner-as-S}
\end{equation}

We compute the second and fourth moments of $S$ without suppressing any
cross terms.  The law \eqref{eq:Y-law} gives
\[
  \E Y_j=0,
  \qquad
  \E Y_j^2=\frac12,
  \qquad
  \E Y_j^4=\frac12.
\]
Independence and centering imply
\begin{equation}
  \E R^2=\frac12\sum_{j=1}^m a_j^2
  =\frac12\norm{a}_2^2.
  \label{eq:R-second}
\end{equation}
In the expansion of $R^4$, the only nonzero monomials are those in
which every index occurs an even number of times.  Hence
\begin{align}
  \E R^4
  &=\sum_{j=1}^m a_j^4\E Y_j^4
    +6\sum_{1\le j<k\le m}a_j^2a_k^2
      \E Y_j^2\E Y_k^2 \notag\\
  &=\frac12\sum_{j=1}^m a_j^4
    +\frac32\sum_{j<k}a_j^2a_k^2 \notag\\
  &=\frac34\left(\sum_{j=1}^m a_j^2\right)^2
    -\frac14\sum_{j=1}^m a_j^4
  \le\frac34\norm{a}_2^4.
  \label{eq:R-fourth}
\end{align}
Multiplying \eqref{eq:R-second} and \eqref{eq:R-fourth} by the
appropriate powers of $\sqrt{2m}$ yields
\begin{equation}
  \E W^2=m\norm{a}_2^2,
  \qquad
  \E W^4\le3m^2\norm{a}_2^4.
  \label{eq:W-moments}
\end{equation}

The random variables $W$ and $\eta$ are independent and symmetric.
Thus their odd moments vanish, and expanding $(W+b\eta)^2$ and
$(W+b\eta)^4$ gives
\begin{align}
  D:=\E S^2
  &=m\norm{a}_2^2+b^2,
  \label{eq:S-second}\\
  \E S^4
  &=\E W^4+6b^2\E W^2+b^4 \notag\\
  &\le3m^2\norm{a}_2^4+6b^2m\norm{a}_2^2+b^4 \notag\\
  &\le3\bigl(m\norm{a}_2^2+b^2\bigr)^2
  =3D^2.
  \label{eq:S-fourth}
\end{align}
The penultimate inequality only replaces $b^4$ by $3b^4$.

On the other hand, $G\preceq mI_m$ and \eqref{eq:v-norm} show that
\begin{equation}
  \norm{v}_2^2=a^\top Ga+b^2
  \le m\norm{a}_2^2+b^2=D.
  \label{eq:D-dominates-v}
\end{equation}
Because $v\ne0$, $D>0$.  Apply \cref{lem:paley-zygmund} to
$Q=S^2$ with $\theta=1/2$.  Using \eqref{eq:S-fourth},
\begin{align}
  \Pp\left(\abs{S}\ge\sqrt{D/2}\right)
  &=\Pp\left(S^2\ge\frac12\E S^2\right) \\
  &\ge\frac14\frac{D^2}{\E S^4}
  \ge\frac1{12}.
  \label{eq:S-small-ball}
\end{align}
By \eqref{eq:D-dominates-v}, the event in
\eqref{eq:S-small-ball} implies
$\abs{S}\ge\norm{v}_2/\sqrt2$.  Finally, $H\ge1$ almost surely, so
$\abs{\ip{Z}{v}}=H\abs{S}\ge\abs{S}$.  This proves
\eqref{eq:proved-small-ball}.
\end{proof}

\subsection{The conic Gaussian width}

\begin{lemma}[A bounded-width convex cone section]
\label{lem:width}
The set $A$ in \eqref{eq:cone-set} satisfies
\begin{equation}
  w(A)\le\frac1{\sqrt{2\pi}}+\sqrt2<2.
  \label{eq:width-bound}
\end{equation}
\end{lemma}

\begin{proof}
Every nonzero vector in $C$ has the form $Ua$ with
$a=(a_1,\ldots,a_m)^\top\ge0$.  Let
$s:=\sum_{j=1}^m a_j>0$ and $\lambda_j:=a_j/s$.  Then
$\lambda=(\lambda_1,\ldots,\lambda_m)$ belongs to the simplex
\[
  \Delta_m:=\left\{\lambda\in\R^m:
  \lambda_j\ge0,\ \sum_{j=1}^m\lambda_j=1\right\}.
\]
Substitution of \eqref{eq:directions} shows
\[
  Ua=\frac{s}{\sqrt{1+\delta^2}}
  \left(e_0+\delta\sum_{j=1}^m\lambda_je_j\right).
\]
After normalization, we obtain the exact parameterization
\begin{equation}
  A=\left\{
  x(\lambda):=
  \frac{e_0+\delta\sum_{j=1}^m\lambda_je_j}
  {\sqrt{1+\delta^2\norm{\lambda}_2^2}}:
  \lambda\in\Delta_m
  \right\}.
  \label{eq:A-parameterization}
\end{equation}

Let $g=(g_0,g_1,\ldots,g_m)\sim N(0,I_{m+1})$ and
$M:=\max_{1\le j\le m}g_j$.  For each $\lambda\in\Delta_m$, put
\[
  N_\lambda:=g_0+\delta\sum_{j=1}^m\lambda_jg_j,
  \qquad
  d_\lambda:=\sqrt{1+\delta^2\norm{\lambda}_2^2}.
\]
Since $\sum_j\lambda_jg_j\le M$, we have
$N_\lambda\le g_0+\delta M$, while $d_\lambda\ge1$.  If
$N_\lambda\ge0$, then
\[
  \frac{N_\lambda}{d_\lambda}
  \le N_\lambda\le(g_0+\delta M)_+.
\]
If $N_\lambda<0$, then
$N_\lambda/d_\lambda<0\le(g_0+\delta M)_+$.  Therefore, pointwise in
$g$,
\begin{equation}
  \sup_{x\in A}\ip{g}{x}
  \le(g_0+\delta M)_+
  \le(g_0)_++\delta\max_{1\le j\le m}\abs{g_j}.
  \label{eq:pointwise-width}
\end{equation}

It remains to bound the Gaussian maximum.  Let
$Q:=\max_j\abs{g_j}$.  For any $t>0$, Jensen's inequality and
$e^{t\abs{x}}\le e^{tx}+e^{-tx}$ give
\begin{align*}
  t\E Q
  &\le\log\E e^{tQ}
  \le\log\left(\sum_{j=1}^m\E e^{t\abs{g_j}}\right)\\
  &\le\log(2m)+\frac{t^2}{2}.
\end{align*}
Dividing by $t$ and taking
$t=\sqrt{2\log(2m)}$ yields
\begin{equation}
  \E\max_j\abs{g_j}\le\sqrt{2\log(2m)}.
  \label{eq:gaussian-max}
\end{equation}
Also, direct integration of the standard normal density gives
$\E(g_0)_+=1/\sqrt{2\pi}$.  Taking expectations in
\eqref{eq:pointwise-width}, using \eqref{eq:gaussian-max}, and recalling
$\delta=1/\sqrt{\log(2m)}$, we conclude that
\[
  w(A)
  \le\frac1{\sqrt{2\pi}}
  +\frac{\sqrt{2\log(2m)}}{\sqrt{\log(2m)}}
  =\frac1{\sqrt{2\pi}}+\sqrt2<2.
\]
\end{proof}

\subsection{A missing-label event}

\begin{lemma}[Annihilation with overwhelming probability]
\label{lem:annihilation}
Let $X\in\R^{n\times(m+1)}$ have independent copies
$Z_1,\ldots,Z_n$ of $Z$ as its rows.  Then
\begin{equation}
  \Pp\left(\inf_{u\in A}\norm{Xu}_2^2=0\right)
  \ge1-(1-2^{-n})^m
  \ge1-\exp(-m2^{-n}).
  \label{eq:annihilation-probability}
\end{equation}
\end{lemma}

\begin{proof}
Write $Y_{ij}$ for the $j$th label in row $i$, and define
\[
  E_j:=\{Y_{1j}=Y_{2j}=\cdots=Y_{nj}=0\}.
\]
The labels are independent over both $i$ and $j$.  Hence
$\Pp(E_j)=(1/2)^n=2^{-n}$, and the events
$E_1,\ldots,E_m$ are mutually independent because they are determined
by disjoint collections of labels.  The row multipliers $H_i$ and signs
$\eta_i$ do not enter these events.

By \cref{lem:labels}, on $E_j$ every entry of $Xu_j$ is zero, so
$Xu_j=0$.  Moreover $u_j\in A$: in
\eqref{eq:A-parameterization}, choose $\lambda_j=1$ and all other
coordinates zero.  It follows that
\[
  \bigcup_{j=1}^mE_j
  \subseteq
  \left\{\inf_{u\in A}\norm{Xu}_2^2=0\right\}.
\]
Independence gives
\[
  \Pp\left(\bigcup_{j=1}^mE_j\right)
  =1-\prod_{j=1}^m\Pp(E_j^c)
  =1-(1-2^{-n})^m.
\]
Finally, $1-x\le e^{-x}$ for $x\ge0$ (because
$\log(1-x)\le-x$ for $0\le x<1$), so
$(1-2^{-n})^m\le\exp(-m2^{-n})$.  This proves
\eqref{eq:annihilation-probability}.
\end{proof}

\subsection{Heavy-tailedness and completion of the proof}

\begin{lemma}[The row law is centered and genuinely heavy-tailed]
\label{lem:heavy-tail}
The vector $Z$ in \eqref{eq:row-law} is centered and has finite
covariance.  Every standard coordinate $Z_k$, $0\le k\le m$, has
infinite fourth moment and a polynomial lower tail.
\end{lemma}

\begin{proof}
The vector $B$ has finite support and is symmetric: replacing
$(Y,\eta)$ by $(-Y,-\eta)$ changes $B$ to $-B$ without changing its
law.  Thus $\E B=0$.  The Pareto law \eqref{eq:H-law} satisfies
$\E H<\infty$ and $\E H^2<\infty$.  Independence therefore gives
$\E Z=\E H\,\E B=0$ and
\[
  \E[ZZ^\top]=\E H^2\,\E[BB^\top],
\]
which is finite entrywise.

Let $F:=\{Y_1=\cdots=Y_m=0\}$.  Then
$\Pp(F)=2^{-m}>0$, and on $F$ we have $B=\eta h$.  Every standard
coordinate of $h$ is nonzero: its zeroth coordinate is
$\delta/\sqrt{m+\delta^2}$, and each remaining coordinate is
$-1/\sqrt{m+\delta^2}$.  By independence and Tonelli's theorem,
for every $k$,
\[
  \E\abs{Z_k}^4
  \ge\E\bigl[H^4\abs{B_k}^4\mathbf{1}_F\bigr]
  =\E H^4\,\Pp(F)\abs{h_k}^4
  =\infty.
\]
More explicitly, for $t\ge\abs{h_k}$,
\begin{align*}
  \Pp(\abs{Z_k}\ge t)
  &\ge\Pp\bigl(F,\ H\abs{h_k}\ge t\bigr)\\
  &=2^{-m}\left(\frac{\abs{h_k}}{t}\right)^3,
\end{align*}
using \eqref{eq:H-law}.  For the matching upper order, define the finite
constant $M_k:=\max\{\abs{b_k}:b\in\operatorname{supp}(B)\}$.  Since
$M_k>0$, for $t\ge M_k$,
\[
  \Pp(\abs{Z_k}\ge t)
  \le \Pp(HM_k\ge t)
  =\left(\frac{M_k}{t}\right)^3.
\]
Thus each coordinate has a two-sided polynomial tail bound of order
$t^{-3}$ up to constants depending on the coordinate and on $m$.
\end{proof}

\begin{proof}[Proof of \cref{thm:counterexample}]
Given $n\ge1$, take $m=4^n$ and $p=m+1$, and use the construction in
\cref{sec:construction}.  A cone generated by finitely many vectors is
polyhedral and closed here; indeed, $U$ is injective and
$C=U\R_+^m$, the image of a closed orthant under a linear isomorphism
onto $\range(U)$.  Thus $C$ is a closed polyhedral convex cone.

The row properties in part (i) follow from
\cref{lem:heavy-tail}.  Part (ii) is exactly
\cref{lem:small-ball}, and part (iii) follows from
\cref{lem:width}.  For part (iv), apply
\cref{lem:annihilation} and use
\[
  m2^{-n}=4^n2^{-n}=2^n.
\]
This yields
\[
  \Pp\left(\inf_{u\in A}\norm{Xu}_2^2=0\right)
  \ge1-\exp(-2^n),
\]
as required.
\end{proof}

\begin{proof}[Proof of \cref{cor:negative-answer}]
Suppose $c_1,c_2>0$ are fixed.  By \cref{thm:counterexample},
$w(A)^2<4$.  Therefore, whenever $n>4c_2/c_1$,
\[
  c_1n-c_2w(A)^2>c_1n-4c_2>0.
\]
On the event in \eqref{eq:theorem-failure}, the left-hand side of the
proposed restricted-eigenvalue inequality is zero while its right-hand
side is strictly positive.  Thus the proposed inequality fails with
probability at least $1-\exp(-2^n)$, and its probability of holding is
at most $\exp(-2^n)$.  This cannot be bounded below by $1-q_n$ for any
$q_n\to0$.
\end{proof}

\section{Interpretation, limits, and possible repairs}
\label{sec:discussion}

\subsection{Why the marginal condition is insufficient}

The small-ball condition is pointwise before the infimum over the random
sample: each deterministic direction has a fixed chance to produce a
measurement of order its Euclidean norm.  In our construction, this is
true uniformly over every direction in the ambient space.  Nevertheless,
the sample is asked to work simultaneously over exponentially many
candidate rays.  Direction $u_j$ depends only on label $Y_j$, and that
label is missing from all $n$ rows with probability $2^{-n}$.  With
$m=4^n$ independent labels, at least one all-zero label occurs with
probability at least $1-e^{-2^n}$.

Euclidean Gaussian width does not see this distribution-specific label
structure.  The rays all lie in an angular neighborhood of $e_0$ whose
radius is of order $1/\sqrt{\log m}$, so the increase from maximizing over
$m$ Gaussian coordinates is exactly canceled by that angular scale.

\subsection{Relation to the small-ball method}

The general small-ball method does not replace the distribution of $Z$
by a standard Gaussian without further assumptions.  A typical
complexity term is the mean empirical width
\begin{equation}
  W_n(A;Z):=
  \E\sup_{u\in A}
  \ip{u}{\frac1{\sqrt n}\sum_{i=1}^n\varepsilon_iZ_i},
  \label{eq:empirical-width}
\end{equation}
where the $\varepsilon_i$ are independent Rademacher signs, independent
of the rows.  Bounds of this type appear in the small-ball framework of
\citet{koltchinskii2015bounding,mendelson2015learning,tropp2015convex}.
For sub-Gaussian rows, increment bounds can compare
$W_n(A;Z)$ with $w(A)$.  The hypothesis
\eqref{eq:global-small-ball} is only a lower-tail statement and provides
no such upper comparison.  Our example is therefore consistent with the
known theory and pinpoints the missing step in the proposed conclusion.

\subsection{What the theorem does and does not settle}

The following qualifications are part of the theorem's scope.
\begin{enumerate}[(a)]
  \item \emph{Uniform constants.}  The counterexample rules out constants
  controlled only by the fixed small-ball parameters, or by other fixed
  universal parameters shared by the family.  It does not assign meaning
  to constants allowed to depend arbitrarily on $n,p,A$, and the full row
  law.
  \item \emph{A high-dimensional family.}  The dimension and row law vary
  with $n$.  This is the standard way to refute a dimension-uniform sample
  complexity theorem.  It is not a claim that one fixed finite-dimensional
  model remains singular as $n\to\infty$.
  \item \emph{Anisotropy.}  The row law is not isotropic.  Isotropy was not
  assumed in the 2015 statement.  Adding it would be a materially stronger
  question.  Merely whitening this construction also changes the geometry
  of the restricted directions and their Euclidean Gaussian width.
  \item \emph{Dependence within a row.}  Coordinates of $Z$ are dependent,
  while rows are independent and identically distributed.  The original
  formulation requires row independence, not coordinate independence.
  \item \emph{Natural restricted geometry.}  The set used here is the
  spherical section of a closed convex cone, not just an arbitrary finite
  subset.  It therefore matches the basic descent-cone form motivating
  restricted eigenvalues, although we do not claim that it is the descent
  cone of a particular named regularizer at a prescribed signal.
\end{enumerate}

A positive theorem with Gaussian-width sample complexity needs some
assumption that connects the row distribution's upper increments to
Euclidean geometry.  Sub-Gaussian increments provide the classical
route.  Other plausible formulations could use the empirical width
\eqref{eq:empirical-width} directly, or combine isotropic normalization
with explicit moment, Orlicz, truncation, or generic-Bernstein controls.
The present theorem says that the marginal small-ball lower bound alone
cannot supply this missing information.

\section{Conclusion}
\label{sec:conclusion}

The small-ball condition controls the probability of a useful
measurement in each fixed direction.  It does not, by itself, control the
complexity of making that statement uniform over a restricted set.  The
explicit construction in this paper separates these two tasks: the
small-ball constants are fixed on the entire ambient space, and the
restricted set is a bounded-width convex-cone section, yet an independent
missing-label mechanism annihilates one restricted ray with probability
tending to one doubly exponentially in the chosen parameterization.

Consequently, the natural distribution-free version of the 2015 COLT
question has a negative answer.  The Euclidean Gaussian width cannot
replace the distribution-dependent empirical width under a marginal
small-ball assumption alone.  A viable positive reformulation must retain
that empirical complexity or add assumptions that compare the row
process to Euclidean Gaussian geometry.  Determining sharp minimal
conditions for such a comparison---especially under isotropic but
non-sub-Gaussian designs---remains a substantive direction for further
work.

\section{Disclosure}
\label{sec:disclosure}
The proof strategy and counterexample were produced by OpenAI’s GPT-5.6 Sol Ultra through Codex in
response to prompts from the author. Codex was also used to revise
the exposition and prepare the LaTeX manuscript. The author selected the
problem, directed the interactions and revisions, and is the sole named author. The AI system is acknowledged as a reasoning and writing tool, not
as an author. This disclosure is not a substitute for independent expert
mathematical review.
\bibliographystyle{plainnat}
\bibliography{references}

\appendix
\section{Supplementary calculations and status notes}
\label{app:supplementary}

\subsection{An explicit formula for the dual frame}

Although only the identities $U^\top V=V^\top U=I_m$ are needed in the
proof, the inverse of the Gram matrix can be written explicitly.  From
\eqref{eq:gram} and the rank-one inverse formula,
\begin{equation}
  G^{-1}
  =\frac{1+\delta^2}{\delta^2}
  \left(I_m-\frac{\one\one^\top}{m+\delta^2}\right).
  \label{eq:explicit-G-inverse}
\end{equation}
To verify this directly, multiply the two unscaled factors:
\begin{align*}
  (\delta^2I_m+\one\one^\top)
  \left(I_m-\frac{\one\one^\top}{m+\delta^2}\right)
  &=\delta^2I_m+\one\one^\top
    -\frac{\delta^2\one\one^\top+m\one\one^\top}
    {m+\delta^2}\\
  &=\delta^2I_m.
\end{align*}
After including the scalar factors from \eqref{eq:gram} and
\eqref{eq:explicit-G-inverse}, the product is $I_m$.

\subsection{The covariance is finite and uniformly nondegenerate}

The proof of \cref{lem:small-ball} also yields a useful normalization
fact.  For $v=Ua+bh$, \eqref{eq:S-second} gives
\[
  \E\ip{B}{v}^2=m\norm{a}_2^2+b^2\ge\norm{v}_2^2.
\]
Since $B$ is centered, this means
\begin{equation}
  \operatorname{Cov}(B)\succeq I_{m+1}.
  \label{eq:B-cov-lower}
\end{equation}
For the Pareto law in \eqref{eq:H-law}, the density on $[1,\infty)$ is
$3t^{-4}$, so
\[
  \E H^2=\int_1^\infty 3t^{-2}\,dt=3.
\]
Independence therefore gives
\begin{equation}
  \operatorname{Cov}(Z)
  =3\operatorname{Cov}(B)\succeq3I_{m+1}.
  \label{eq:Z-cov-lower}
\end{equation}
Thus the counterexample is not caused by a zero-variance ambient
direction.  Its covariance is finite and positive definite, but highly
anisotropic.

\subsection{Robustness to an absolute-width convention}

Some authors use the symmetrized quantity
\[
  w_{\mathrm{abs}}(A):=
  \E\sup_{u\in A}\abs{\ip{g}{u}}.
\]
The same conic set has bounded width under this convention.  Indeed,
from \eqref{eq:A-parameterization},
\begin{align*}
  \sup_{u\in A}\abs{\ip{g}{u}}
  &\le\sup_{\lambda\in\Delta_m}
  \abs{g_0+\delta\textstyle\sum_j\lambda_jg_j}\\
  &\le\abs{g_0}+\delta\max_j\abs{g_j}.
\end{align*}
Taking expectations and using \eqref{eq:gaussian-max},
\begin{equation}
  w_{\mathrm{abs}}(A)
  \le\sqrt{\frac2\pi}+\sqrt2,
  \label{eq:absolute-width-bound}
\end{equation}
an absolute constant.  Hence the obstruction is not an artifact of the
one-sided convention in \eqref{eq:gaussian-width}.

\subsection{Why whitening does not preserve the counterexample's width}

Because \eqref{eq:Z-cov-lower} gives a finite positive-definite
covariance, one may ask whether whitening turns the example into an
isotropic counterexample.  Whitening changes both the row distribution
and the Euclidean geometry against which Gaussian width is measured.  In
the finite-support base construction, the dual directions are stretched
very differently from their common $e_0$ component.  After such a linear
change of variables, the corresponding restricted directions are no
longer an exponentially large family inside an angular neighborhood of
radius $1/\sqrt{\log m}$.  Therefore the bounded-width estimate in
\cref{lem:width} is not invariant under whitening.  The present proof
does not settle a formulation that assumes isotropy from the outset.

\subsection{Literature-status protocol}
\label{app:status}

The source problem is the four-page COLT note of
\citet{banerjee2015open}.  The following checks were performed through
August 16, 2026:
\begin{enumerate}[(1)]
  \item exact-title searches for the open-problem note together with the
  terms \emph{solution}, \emph{counterexample}, \emph{Gaussian width}, and
  \emph{small ball};
  \item searches for later works by the three authors that cite the
  Gaussian-width formulation; and
  \item comparison with later small-ball and heavy-tailed generalized
  Lasso results, including
  \citet{tropp2015convex,mendelson2017extending,genzel2022generic}.
\end{enumerate}
These searches located the original problem and later results using
distribution-adapted empirical or mixed-tail complexities, but no paper
claiming an explicit resolution of the exact small-ball-only question.
This search record cannot prove novelty: citation databases may be
incomplete, terminology may differ, and unpublished work may exist.
Before submission, the appropriate next steps are a formal MathSciNet and
zbMATH citation search, a current citation-index audit, and contact with
the original authors.

The result proved here should therefore be described as an
\emph{apparently unpublished negative resolution of the natural uniform
formulation}, not as a certified first solution.

\end{document}